\documentclass{amsart}

\usepackage{graphicx}
\usepackage{amsmath}
\usepackage{amssymb}
\usepackage{bbm}
\usepackage{difftrees}
\def\RR{\mathbb R}

\newcommand{\ones}{\mathbbm{1}}
\newcommand{\oneslm}[1]{\mathbbm{1}^{\{#1\}}}

\newcommand{\diag}[1]{\hbox{{\rm diag}}( #1 )}

\def\R#1{$(\ref{#1})$}
\def\D{\,{\rm d}}

\newcommand{\de}{\partial}

\newcommand{\Bl}[1]{B^{\{#1\}}}
\newcommand{\Clm}[1]{C^{\{#1 \}}}
\newcommand{\Dlm}[1]{D^{\{#1\}}}
\newcommand{\etalm}[1]{\eta^{\{#1\}}}
\newcommand{\zetalm}[1]{\zeta^{\{#1\}}}
\newcommand{\rl}[1]{r^{\{#1\}}}
\newcommand{\myss}[1]{s^{\{#1\}}}
\renewcommand{\aa}[1]{a^{\{ #1\}}}
\newcommand{\waa}[1]{\widehat{a}^{\{ #1\}}}
\newcommand{\bb}[1]{b^{\{ #1\}}}
\newcommand{\cc}[1]{c^{\{ #1\}}}
\renewcommand{\AA}[1]{A^{\{ #1\}}}
\newcommand{\wAA}[1]{\widehat{A}^{\{ #1\}}}
\newcommand{\QQ}[1]{Q^{\{ #1\}}}
\newcommand{\PP}[1]{P^{\{ #1\}}}
\newcommand{\QQdot}[1]{\dot{Q}^{\{ #1\}}}
\newcommand{\PPdot}[1]{\dot{P}^{\{ #1\}}}
\newcommand{\NT}{\mathrm{NT}}

\newcommand{\one}{\mathbbm{1}}
\newcommand{\LL}{\mathcal{L}}
\newtheorem{theorem}{Theorem}[section]
\newtheorem{coroll}{Corollary}[theorem]
\newtheorem{lemma}[theorem]{Lemma}
\newtheorem{prop}[theorem]{Proposition}
\newtheorem{definition}{Definition}[section]

\newenvironment{meqn}
{\arraycolsep=1.4pt
  
  \begin{array}{rcl}}
  {\end{array}}

\begin{document}
\author[A. Zanna]{Antonella Zanna$^\dag$}
\thanks{$^\dag$ Matematisk institutt, Universitetet i Bergen, Norway,
  email: \texttt{Antonella.Zanna@uib.no}}
\date{\today}
\title{Discrete Variational Methods and Symplectic Generalized Additive Runge--Kutta Methods}

\begin{abstract}
  We consider a Lagrangian system $L(q,\dot q) =
  \sum_{l=1}^{N}L^{\{l\}}(q,\dot q)$, where the $q$-variable is
  treated by a Generalized Additive Runge--Kutta (GARK)
  method. Applying the technique of discrete variations, we show how
  to construct symplectic schemes. Assuming the diagonal methods for
  the GARK given, we present some techinques for constructing the
  transition matrices. We address the problem of the order of the
  methods and discuss some semi-separable and separable problems,
  showing some interesting constructions of methods with non-square
  coefficient matrices.
\end{abstract}

\maketitle

\section{Introduction}
\label{sec:introduction}
In this paper we are interested in studying a family of variational
method that are obtained when considering a system with a
Lagrangian
\begin{displaymath}
  L(q, \dot q) = \sum_{l=1}^{N}L^{\{l\}}(q,\dot q), \qquad q \in \RR^d,
\end{displaymath}
or a Hamiltonia system, split in elementary Hamiltonian systems,
\begin{displaymath}
  H(q, p) = \sum_{l=1}^{N}H^{\{l\}}(q,p), \qquad q,p\in \RR^d,
\end{displaymath}
where each term is treated by a different method of Runge--Kutta (RK) type.
The idea of treating different terms with different methods is by no
means new. Additive Runge--Kutta (ARK) methods were introduced already
in the 80's \cite{cooper83ark} to deal with stiff ODEs, where the
stiff term would be treated by a different Runge--Kutta method than
the non-stiff term, typically using an explicit method for the
non-stiff part and an implicit one for the stiff-part. In the
mid-90's, these methods were studied in detail from the Hamiltonian
viewpoint, and order conditions and conditions for symplecticity were
established \cite{araujo97smb}. 
The use of additivity, especially in the context of DAEs was studied
in \cite{jay98spark}, and later in \cite{tanner18gark}, the latter
especially in the context of the formalism of Generalized Additive
Runge-Kutta (GARK) methods introduced in \cite{sandu15gark}. 

Parallel to the Hamiltonian approach, there is the Lagrangian approach,
popular in the  community of computational mechanics
\cite{marsden_west_2001} and optimal control \cite{ober-blobaum11dma}. In the Lagrangian setting, the action
integral is discretized by an appropriate quadrature, the variable is
replaced by an appropriate polynomial interpolant and discrete
variational equations are derived. As the discrete variational
equations are essentially the same as generating functions, the resulting 
methods are automatically symplectic. 
Recently, a splitting of the Lagrangian, where each term was treated
by a different method, was used in the context of higher order
variational integrators for dynamical systems with holonomic
constraints \cite{wenger2017caa} and in order to devise mixed order 
integrators for systems with multiple scales \cite{wenger2016vio}. The
order analysis of these methods is not straightforward.

It is known that there is an equivalence between symplectic
Partitioned RK methods and some discrete variational methods, see
\cite{hairer06gni, marsden_west_2001}, but it is not known whether more complicated
variational methods can be written in a RK-type formalism. The
advantage a RK formalism is that it would make order analysis of the
methods considerably easier, as the order analysis of RK and PRK methods is
well understood (see \cite{hairer06gni} chapter III and references therein).

In a recent paper \cite{zanna19sps}, we introduced a family of Runge--Kutta methods
of additive type for highly oscillatory problems.
The methods were derived from a variational formulation, using
different quadrature formulas for different parts of the Lagrangian. 
In this paper we identify those methods as a particular case of
partitioned symplectic GARK methods, that are, in turn, GARK methods.
The main idea is that the $q$-variables are treated by a GARK
method, thus generating a discrete Lagrangian that uses a different RK
for each part rather than the same RK for all terms. Performing discrete 
variations, we arrive to a partioned GARK method that is automatically
symplectic. The GARK framework allows us to significantly simplify the
formulation of the methods, that can be written in a ready-to-use
formalism,  and the order analysis, as only
known algebraic relations for the coefficients need to be verified.

The paper is organized as follows. We commence by reviewing the
formalism of GARK methods and the equivalence with ARK methods,
see \S~\ref{sec:gark-methods}. In Section~\ref{sec:vari-deriv} we use
GARK methods for the splitted Lagrangian to obtain variational
equations and symplectic partitioned GARK methods. In
Section~\ref{sec:gener-form-meth} we present the general form of the
methods for both Lagrangian and Hamiltonian problems, the latter split
in Hamiltonian sub-problems. In Section~\ref{sec:order} we review some
results about the order conditions for GARK methods and reformulate
them in a simple and elegant way. In section~\ref{sec:transfer} we
introduce two main techniques that can be used to construct the
transfer matrices between the different methods. The transfer matrices
constructed in this way can be used as generators for other transfer
matrices. As the transfer matrices are not unique (in facts there are
infinitely many of those), this gives the possibility of tuning the
methods to obey qualitative properties which would be hard to
obey in the classical setting. For instance, in \cite{zanna19sps} we
constructed higher order symplectic methods that were
P-stable using a diagonal method that was not P-stable. In
Section~\ref{sec:special-case} we consider the special case when the
number of splitting terms is $N=2$. We consider a semi-separable case
$L(q,\dot q)=L^{\{1\}}(q, \dot q) +  L^{\{2\}}(q)$, and a fully
separable case,  $L(q,\dot q)=L^{\{1\}}(\dot q) +  L^{\{2\}}(q)$,
together with the corresponding Hamiltonian cases. For this latter
case, we show, as an example of the new possibilites opened by this
framework, a fourth order symplectic method with three stages for $q$
and two stages for $p$, thus having rectangular coefficient matrices. The method
 is constructed from the Gauss--Legendre and Gauss--Lobatto of order
four but is implicit only in the second stage $Q_2$ for $q$.
Finally, Section~\ref{sec:concl} is devoted to some concluding remarks.

\section{GARK methods}
\label{sec:gark-methods}
A GARK method  for the problem
\begin{equation}
  \label{eq:5}
  \dot x = f(t,x) = \sum_{l=1}^N f^{\{l\}}(t,x)
\end{equation}
reads
\begin{eqnarray*}
  X^{\{l\}}_i &=& x_n + h \sum_{m=1}^N \sum_{j=1}^{\myss{m}}
                  \aa{l,m}_{i,j} f^{\{m\}}( t_j^{\{m\}}, X_j^{\{ m\}}),
                  \quad  i=1, \ldots \myss{l}, \quad l=1, \ldots, N,
  \\
  x_{n+1} &=& x_n + h \sum_{l=1}^N \sum_{j=1}^{\myss{l}} \bb{l}_j
              f^{\{l\}}(t_j^{\{l\}}, X_j^{\{j\}}),
\end{eqnarray*}
with $t_j^{\{l\}} = t_n + \cc{l}_j h$, and the corresponding
generalized Butcher tableau of coefficients
\begin{equation}
  \label{eq:gen-gark}
  \begin{array}{c|ccc}
    \cc{1}&\AA{1,1} &\cdots& \AA{1,N}\\
    \vdots & \vdots & \ddots & \vdots\\
    \cc{N} & \AA{N,1} &\cdots & \AA{N,N}\\ \hline 
          &\bb1 & \cdots & \bb{N}
  \end{array}
\end{equation}
\cite{sandu15gark}.
The diagonal blocks of the type $(\AA{l,l}, \bb{l}, \cc{l})$,
$l=1, \dots, N$, are usually chosen as some standard RK methods, while the off-diagonal
blocks $\AA{l,m}$, $l\not =m$, are coupling (or transfer)
coefficients.

\subsection{The equivalence of ARK and GARK}
The formalism of the GARK methods is equivalent to that of
Additive RK methods, and RK method, providing a unified approach. For
instance, when $N=2$,  the  ARK method
\begin{displaymath}
  \begin{array}{c|cc}
    c&\AA1&\AA2\\ \hline
      & \bb1& \bb2
  \end{array} \quad \hbox{is equivalent to the GARK method} \quad
  \begin{array}{c|cc}
    c & \AA1 & \AA2\\
    c & \AA1 & \AA2\\ \hline
       & \bb1 & \bb2
  \end{array}.
\end{displaymath}
Vice versa, the GARK method
\begin{displaymath}
  \begin{array}{c|cc}
    \cc{1} & \AA{1,1} & \AA{1,2}\\
    \cc{2} & \AA{2,1} & \AA{2,2}\\ \hline
       & \bb1 & \bb2
  \end{array}\quad \hbox{is equivalent to the ARK method} \quad
  \begin{array}{c|cc}
    \begin{bmatrix}
      \cc{1}\\
      \cc{2}
    \end{bmatrix}
     &
    \begin{bmatrix}
      \AA{1,1} & 0\\ \AA{2,1}&0
    \end{bmatrix}
    &
    \begin{bmatrix}
      0 & \AA{1,2} \\ 0 & \AA{2,2}
    \end{bmatrix} \\ \hline
     &
     \begin{bmatrix}
       \bb{1} & 0
     \end{bmatrix}
     &
     \begin{bmatrix}
       0 & \bb2
     \end{bmatrix}
  \end{array}
\end{displaymath}
The advantage of the GARK fomulation is that it clarifies the coupling
between the various methods, in addition to eliminating zero
quadrature weights in the ARK formalism, hence the analysis of
special cases.

\subsection{Partitioned Runge--Kutta methods are GARK methods}
\label{sec:prk-are-gark}
Consider a generic problem with a partitioning of the variables of the type
\begin{eqnarray*}
  \label{eq:2}
  \dot q &=& v(q,p) \\
  \dot p &=& f(q,p),
\end{eqnarray*}
and a Partitioned Runge--Kutta method $(\AA1, \bb1, \cc1)$, $(\AA2,
\bb2,\cc2)$. It is usual to choose $\bb2=\bb1=b$ and $\cc2=\cc1=c$.
Let  variable $y =[q,p]^T$ and the splitting
\begin{displaymath}
  \dot y = \mathcal{F}(y) = \mathcal{F}^{\{1\}}(y) + \mathcal{F}^{\{2\}}(y) =
  \begin{bmatrix}
    v(q,p)\\0
  \end{bmatrix} + 
\begin{bmatrix}
    0\\f(q,p)
  \end{bmatrix}.
\end{displaymath}
It is easy to see that the PRK method corresponds to the GARK method
\begin{displaymath}
  \begin{array}{c|cc}
    c& \AA1&\AA2\\
    c& \AA1 & \AA2\\ \hline
     & b & b
  \end{array}.
\end{displaymath}
In facts, by virtue of the fact that the $q$-part of $\mathcal{F}^{\{2\}}$ is
zero, one has $\QQ1_i = \QQ2_i$ for all the stages of the
methods. By a similar argument, $\PP1_i=\PP2_i$ and the statement
follows. In particular, if the problem is Hamiltonian, that is
\begin{displaymath}
  v(q,p) =\frac{\de H(q,p)}{\de p}, \qquad f(q,p) = - \frac{\de
    H(q,p)}{\de q}
\end{displaymath}
and if the PRK is symplectic, then the above GARK becomes ($\AA{1} = A$)
\begin{displaymath}
  \begin{array}{c|cc}
    c& A&\widehat A\\
    c& A & \widehat A\\ \hline
     & b & b
  \end{array}, 
\end{displaymath}
with the symplectic condition $ b_i b_j = b_i \widehat a_{i,j} + b_j a_{j,i}$.

\section{Variational derivation of Partitioned Symplectic GARK methods}
\label{sec:vari-deriv}
Assume a Lagrangian $L(q,\dot q)=\sum_{l=1}^N L^{\{l\}}(q, \dot q)$
given. We consider a discrete action approximation
\begin{equation}
  \label{eq:split_lagr}
  \int_0^hL(q, \dot q) \approx h \sum_{l=1}^N\sum_{k=1}^{\myss{l}} \bb{l}_k
  L^{\{l\}}(\QQ{l}_k, \QQdot{l}_k),
\end{equation}
where we assume that the $q$-variables are resolved by a GARK method
\begin{eqnarray}
  \label{eq:8}
  \QQ{l}_i &=& q_0 + h \sum_{m=1}^N \sum_{j=1}^{\myss{m}} \aa{l,m}_{i,j} \QQdot{m}_j, \qquad i=1,
             \ldots, \myss{l}, \quad l=1,\ldots, N\\
  q_1 &=& q_0 + h \sum_{l=1}^N \sum_{i=1}^{\myss{l}}\bb{l}_i \QQdot{l}_i,
\end{eqnarray}
corresponding to the GARK tableau for the $q$ variables,
\begin{displaymath}
  Q: \qquad
  \begin{array}{c|ccc}
    \cc1&\AA{1,1}&\cdots &\AA{1,N}\\
    \vdots & \vdots & \ddots &\vdots\\
     \cc{N}  &\AA{N,1}&\cdots &\AA{N,N}\\ \hline 
           & \bb1 & \cdots & \bb{N}
  \end{array}
\end{displaymath}
with $\bb{l}_i\not=0$ for $i=1,\ldots, \myss{l}$, $l=1,\ldots, N$.
We construct the \emph{augmented discrete Lagrangian} by taking
\begin{displaymath}
  L_\lambda(q_0, q_1) = h \sum_{l=1}^N\sum_{k=1}^{\myss{l}} \bb{l}_k
  L^{\{l\}}(\QQ{l}_k, \QQdot{l}_k) -\lambda ( q_1-q_0 - h \sum_{l=1}^N \sum_{i=1}^{\myss{l}}\bb{l}_i \QQdot{l}_i).
\end{displaymath}
The augmentation of the Lagrangian takes care of the linear dependence
in the variables. Proceding in a manner similar to the derivation of
symplectic PRK, we use the $\QQdot{m}_j$, $j = 1, \ldots,
\myss{m}, \,m=1, \ldots, N$ as the principal variables. Taking
derivative with respect to $\QQdot{m}_j$, we obtain the following $\myss1+\cdots+ \myss{N}$ conditions
\begin{equation}
  \label{eq:des}
  h \sum_{l=1}^{N}\sum_{k=1}^{\myss{l}} \bb{l}_k (\PPdot{l}_k\frac{\de
  \QQ{l}_k}{\de \QQdot{m}_j} + \PP{l}_k \frac{\de
  \QQdot{l}_k}{\de \QQdot{m}_j}) + \lambda h \bb{m}_j =0, \qquad 
    j =1,\ldots,\myss{m},\quad m=1,\ldots, N,
\end{equation}
where we have denoted
\begin{align*}
  \PPdot{l}_k = \frac{\de}{\de {q}} L^{\{l\}}(\QQ{l}_k,
  \QQdot{l}_{k}), \qquad
    \PP{l}_k = \frac{\de}{\de \dot{q}} L^{\{l\}}(\QQ{l}_k,
      \QQdot{l}_{k}).
\end{align*}
Using the fact that $\frac{\de
  \QQ{l}_k}{\de \QQdot{m}_j} = h \aa{l,m}_{k,j} I$ and substituting
into \R{eq:des}, we obtain
\begin{equation}
  \label{eq:des_final}
  h \sum_{l=1}^{N}\sum_{k=1}^{\myss{l}} \bb{l}_k (\PPdot{l}_k h \aa{l,m}
  _{k,j} ) + h^2 \bb{m}_j\PP{m}_j + \lambda h \bb{m}_j =0, \qquad 
    j =1,\ldots,\myss{m},\quad m=1,\ldots, N.
\end{equation}
The symplectic numerical method is obtained by considering the
discrete Euler equations for $L_\lambda(q_0, q_1)$
\begin{displaymath}
  \frac{\de }{\de q_0} L_\lambda(q_0, q_1) + \frac{\de}{\de q_1}
  L_\lambda(q_{-1}, q_0) =0,
\end{displaymath}
(variations are zero at the endpoint of integration) and eliminating the Lagrange multiplier $\lambda$, giving a two-step
type method in $q_{-1}, q_0,q_1$. The method can be reduced to a one-step method by using the the Legendre transform
\begin{displaymath}
  p_0 = - \frac{\de }{\de q_0} L_\lambda(q_0,q_1) , \qquad p_1 = \frac{\de}{\de q_1} L_\lambda(q_0,q_1),
\end{displaymath}
and we shall consider the latter approach.
We have
\begin{displaymath}
  \frac{\de L_{\lambda}}{\de {q_0}} = h \sum_{l=1}^N\sum_{k=1}^{\myss{l}} \bb{l}_k(\PPdot{l}_k
  \frac{\de \QQ{l}_k}{\de q_0} + \PP{l}_k \frac{\de \QQdot{l}_k}{\de q_0} )
  + \lambda (I + h \sum_{l=1}^N\sum_{k=1}^{\myss{l}}\bb{l}_k \frac{\de
    \QQdot{l}_k}{\de q_0}),
\end{displaymath}
and, taking into account that
\begin{displaymath}
  \frac{\de \QQ{l}_k}{\de q_0} = I + h \sum_{m=1}^N 
\sum_{j=1}^{\myss{m}} \aa{l,m}_{k,j} \frac{\de \QQdot{m}_j}{\de q_0},
\end{displaymath}
together with \R{eq:des_final}, we obtain
\begin{equation}
  \label{eq:p0}
  p_0 = -h \sum_{l=1}^N\sum_{k=1}^{\myss{l}} \bb{l}_k \PPdot{l}_k -\lambda . 
\end{equation}
Similarly, 
\begin{displaymath}
  \frac{\de L_{\lambda}}{\de {q_1}} = h \sum_{l=1}^N\sum_{k=1}^{\myss{l}} (\PPdot{l}_k
  \frac{\de \QQ{l}_k}{\de q_1} + \PP{l}_k \frac{\de \QQdot{l}_k}{\de q_1} )
  -\lambda I,
\end{displaymath}
and, again, expanding in terms of $\frac{\de \QQdot{l}_k}{\de q_1}$
and using \R{eq:des_final}, we obtain
\begin{displaymath}
  p_1 = -\lambda I,
\end{displaymath}
from which,
\begin{displaymath}
  p_1 = p_0 + h \sum_{l=1}^N\sum_{k=1}^{\myss{l}} \bb{l}_k \PPdot{l}_k.
  \end{displaymath}
  The internal stages for the $\PP{l}_i$s can now be retrieved
  combining \R{eq:des_final} and \R{eq:p0}:
  \begin{eqnarray}
    \PP{l}_j &=& p_0 + h \sum_{m=1}^N \sum_{k=1}^{\myss{m}}
    (\bb{m}_k- \frac{\bb{m}_{k}}{\bb{l}_j}\aa{m,l}_{k,j}) \PPdot{m}_k \quad j=1,\ldots,
                 \myss{l} \label{eq:Pjl} 
  \end{eqnarray}
thus obtaining a Partitioned Symplectic GARK method
\begin{equation}
  \label{eq:psymplecticGARK}
  Q: \quad \begin{array}{c|ccc}
             \cc1&\AA{1,1}& \cdots &\AA{1,N}\\
             \vdots & \vdots & \ddots & \vdots\\
    \cc{N}&\AA{N,1}&\cdots &\AA{N,N}\\ \hline
           & \bb1 & \cdots &\bb{N}
  \end{array}, \qquad P: \quad \begin{array}{c|ccc}
             \cc1&\wAA{1,1}& \cdots &\wAA{1,N}\\
             \vdots & \vdots & \ddots & \vdots\\
    \cc{N}&\wAA{N,1}&\cdots &\wAA{N,N}\\ \hline
           & \bb1 & \cdots &\bb{N}
  \end{array},
\end{equation}
where\footnote{For consistency, we will require $\cc{l} = \AA{l,m}
  \oneslm{m}$.}
\begin{equation}
  \label{eq:wAhat}
  \wAA{l,m} = (\oneslm{l,m} - (B^{\{l\}})^{-1}
    (\AA{m,l})^T) B^{\{m\}}, \qquad  l,m =1,\ldots, N,
  \end{equation}
 $B^{\{k\}}=\diag{\bb{k}}$ are required to be invertible, and $\oneslm{l,m}=\one_{\myss{l}\times{\myss{m}}}$ is the
 matrix of ones of dimension ${\myss{l}\times{\myss{m}}}$.
 It is easy to see that
 \begin{displaymath}
   \AA{m,l} =  (\oneslm{l,m} - (B^{\{m\}})^{-1}
    (\wAA{l,m})^T) B^{\{l\}}, \qquad  l,m =1,\ldots, N,
  \end{displaymath}
From the relation \R{eq:wAhat} above, it is clear that, once the GARK method for
the $q$-variables is chosen, the matrices $\wAA{l,m}$,
$l,m=1,\ldots,N$, exist and are uniquely defined, as long as the
weights in $\bb{l}, \bb{m}$ are all nonzero. 
Moreover the correspondence
is one to one.  As the $p$-variables are conjugate to the $q$
variables, the one-to-one relation \R{eq:wAhat}
 justifies the following definition.

 \begin{definition}[Symplectic conjugate GARK methods]
The couple of GARK methods $(\AA{l,m}, \bb{m},\cc{l})$,  $(\wAA{l,m},
\bb{m},\cc{l})$, $l,m=1,\ldots, N$, with $\wAA{l,m}$ defined as in
\R{eq:wAhat}  and
$\bb{l}_i\not=0$, $i=1,\ldots, \myss{l}$, 
 $l=1,\ldots, N$, will be called a \emph{symplectic conjugate} GARK
 method and denoted by $(\AA{l,m}, \wAA{l,m}, \bb{m}, \cc{l})$, $l,m=1,\ldots, N$.
\end{definition}
The relation \R{eq:wAhat} can be written as a generalization of the well known
symplectic condition 
  \begin{displaymath}
    \bb{l}_i\bb{m}_j = \bb{l}_i \waa{l,m}_{i,j} + \bb{m}_j
    \aa{m,l}_{j,i} \qquad \forall i=1,\ldots, \myss{l}, \quad j=1,
    \ldots, \myss{m}, \qquad \forall l,m=1, \ldots, N.
  \end{displaymath}
 Note  that the diagonal blocks of the $Q$ and $P$ variables are
 precisely symplectic conjugate PRK pairs $(\AA{l,l}, \wAA{l,l},\bb{l},\cc{l})$.
 
In what follows, we will use consistently the wide hat notation $\widehat{~~}$
to denote a matrix that is constructed using \R{eq:wAhat}.

\section{General format of the methods: the Lagrangian and Hamiltonian setting}
\label{sec:gener-form-meth}
The Symplectic Partitioned GARK methods read as follows:
\begin{equation}
    \label{eq:1}
\begin{meqn}
  \QQ{l}_i &=& \displaystyle q_0 + h \sum_{m=1}^N \sum_{j=1}^{\myss{m}} \aa{l,m}_{i,j} \QQdot{m}_j, \qquad i=1,
             \ldots, \myss{l}, \quad l=1,\ldots, N\\
  q_1 &=& \displaystyle q_0 + h \sum_{l=1}^N \sum_{i=1}^{\myss{l}}\bb{l}_i 
          \QQdot{l}_i, \\
 \PP{l}_i &=& \displaystyle p_0 + h \sum_{m=1}^N \sum_{j=1}^{\myss{m}} \waa{l,m}_{i,j}\PPdot{m}_j \quad i=1,\ldots,
                 \myss{l}, \quad l=1,\ldots, N \\
   p_1 &=& \displaystyle p_0 + h \sum_{l=1}^N\sum_{k=1}^{\myss{l}} \bb{l}_k \PPdot{l}_k.
\end{meqn}
\end{equation}
where $\PP{l} = \frac{\de }{\de \dot q} L^{\{l\}} (\QQ{l},
\QQdot{l})$ (Legendre transform) and $\wAA{l,m}$ defined as in \R{eq:wAhat}. Assuming the latter to be
invertible, we can solve for $\QQdot{l}$ to obtain
\begin{displaymath}
  \QQdot{l} = V^{\{l\}} (\QQ{l},\PP{l}),
\end{displaymath}
so that
\begin{displaymath}
  \PPdot{l} =  \frac{\de }{\de q} L^{\{l\}} (\QQ{l},
\QQdot{l}) = F^{\{l\}}(\QQ{l}, \PP{l}).
\end{displaymath}
Thus, the method can be written as a one step method $(q_0,p_0) \to
(q_1,p_1)$ in the form
\begin{equation}
    \label{eq:spgqrk}
\begin{meqn}
  \QQ{l}_i &=& \displaystyle q_0 + h \sum_{m=1}^N \sum_{j=1}^{\myss{m}}
  \aa{l,m}_{i,j} V^{\{m\}}(\QQ{m}_j, \PP{m}_j) \qquad i=1,
             \ldots, \myss{l}, \quad l=1,\ldots, N\\
  q_1 &=& \displaystyle q_0 + h \sum_{l=1}^N \sum_{i=1}^{\myss{l}}\bb{l}_i 
         V^{\{l\}} (\QQ{l}_i, \PP{l}_i) \\
 \PP{l}_i &=& \displaystyle p_0 + h \sum_{m=1}^N \sum_{j=1}^{\myss{m}} \waa{l,m}_{i,j}F^{\{m\}}(\QQ{m}_j,\PP{m}_j) \quad i=1,\ldots,
                 \myss{l}, \quad l=1,\ldots, N \\
   p_1 &=& \displaystyle p_0 + h \sum_{l=1}^N\sum_{i=1}^{\myss{l}} \bb{l}_i F^{\{l\}}(\QQ{l}_i,\PP{l}_i).
\end{meqn}
\end{equation}
Note that the above formulation yields a symplectic method in the Hamiltonian setting,
when
\begin{displaymath}
  H(q,p) =\sum_{l=1}^N H^{\{l\}} (q,p), \qquad V^{\{l\}}(q,p) = \frac{\de
  H^{\{l\}} }{\de p}(q,p), \quad F^{\{l\}}(q,p) = -\frac{\de
  H^{\{l\}} }{\de q}(q,p).
\end{displaymath}

Next, consider a partitioning of the system
\begin{displaymath}
  \dot y = \mathcal{F}(y) = \mathcal{F}^{\{1\}}(y) + \cdots + \mathcal{F}^{\{2N\}}(y) =
  \sum_{l=1}^N\begin{bmatrix}
    V^{\{l\}}(q,p)\\0
  \end{bmatrix} + \sum_{l=1}^N
\begin{bmatrix}
    0\\F^{\{l\}}(q,p)
  \end{bmatrix}.
\end{displaymath}

We consider first the case when $N=2$. We have a splitting in four
additive vector fields,
\begin{displaymath}
  \begin{bmatrix}
    V^{\{1\}}(q,p)\\0
  \end{bmatrix} + \begin{bmatrix}
    V^{\{2\}}(q,p)\\0
  \end{bmatrix}  + 
\begin{bmatrix}
    0\\F^{\{1\}}(q,p)
  \end{bmatrix} + \begin{bmatrix}
    0\\F^{\{2\}}(q,p)
  \end{bmatrix},
\end{displaymath}
where, under the assumption that $F = \nabla H$ is Hamiltonian, we
take $H(q,p) = H^{\{1\}}(q,p) + H^{\{2\}}(q,p)$, and
\begin{displaymath}
   V^{\{l\}}(q,p) = \frac{\de H^{\{l\}}(q,p)}{\de p}, 
   \qquad  F^{\{l\}}(q,p) = -\frac{\de H^{\{l\}}(q,p)}{\de p}, \quad l=1,2.
\end{displaymath}
The method \R{eq:psymplecticGARK} coefficients $\wAA{l,m}$ satisfying
\R{eq:wAhat} is then equivalent to a (symplectic) GARK method
\begin{equation}
  \label{eq:6}
  \begin{array}{c|cccc}
    \cc1&\AA{1,1}&\AA{1,2}&\wAA{1,1}&\wAA{1,2}\\
    \cc2&\AA{2,1} & \AA{2,2} & \wAA{2,1} & \wAA{2,2}\\
    \cc1&\AA{1,1}&\AA{1,2}&\wAA{1,1}&\wAA{1,2}\\
    \cc2&\AA{2,1} & \AA{2,2} & \wAA{2,1} & \wAA{2,2} \\
    \hline
     & \bb1 & \bb2 & \bb1 & \bb2
  \end{array}.
\end{equation}
Also in this scheme we have redundancy, as $\QQ{3}_i = \QQ1_i$ for $i=1,\ldots, \myss1$ and
$\QQ4_i= \QQ2_i$ for $i=1,\ldots, \myss2$, being the $3$rd, $4$th additive
terms of the $q$-vector fields zero; similarly, one has $\PP1_i = \PP3_i$,
$i=1,\ldots, \myss1$ and $\PP2_i = \PP4_i$, $i=1, \ldots, \myss2$, being
the $1$st, $2$nd part of the $p$-vector field zero.

A generic $N$-terms symplectic partitioned GARK method
\R{eq:psymplecticGARK}  can, in turn, be written as
a GARK method (with lots of redundancy) in a similar manner, by taking
$2$ copies of each set of  $q$ and $p$ variables, so that $\QQ{l+N}_i=
\QQ{l}_i$ and $\PP{l}_i=\PP{l+N}_i$, $i=1,\ldots, \myss{l}$, $l=1,\ldots,
N$, resulting in a $2N$ \emph{symplectic GARK method}
\begin{equation}
  \label{eq:varGARK}
  \begin{array}{c|cccccc}
    \cc1&\AA{1,1}&\cdots  & \AA{1,N}&\wAA{1,1}&\cdots & \wAA{1,N}\\
    \vdots &\vdots &\ddots & \vdots & \vdots &\ddots & \vdots\\
    \cc{N}&\AA{N,1} & \cdots & \AA{N,N} & \wAA{N,1} & \cdots & \wAA{N,N}\\
    \cc1&\AA{1,1}&\cdots  & \AA{1,N}&\wAA{1,1}&\cdots & \wAA{1,N}\\
    \vdots &\vdots &\ddots & \vdots & \vdots &\ddots & \vdots\\
    \cc{N}&\AA{N,1} & \cdots & \AA{N,N} & \wAA{N,1} & \cdots & \wAA{N,N}\\
    \hline
     & \bb1 & \cdots & \bb{N} & \bb1 & \cdots &\bb{N}
  \end{array}.
\end{equation}
where $\wAA{l,m}$ is defined as in \R{eq:wAhat}.
The benefit of the above formulation is that we can take advantage of
the existing order analysis already developed for GARK methods.


\section{Order conditions for GARK methods}
\label{sec:order}
In this section we consider a generic GARK method,
\begin{equation}
  \label{eq:GARK}
  \begin{array}{c|ccc}
    \cc1 & \AA{1,1} & \cdots &\AA{1,N}\\
    \vdots & \vdots & \ddots & \vdots\\
    \cc{N} & \AA{N,1}& \cdots & \AA{N,N}\\ \hline
         & \bb{1} & \cdots & \bb{N}
  \end{array}
\end{equation}
for the differential equation
\begin{displaymath}
  \dot y = \mathcal{F}(y)= \sum_{l=1}^N \mathcal{F}^{\{l\}}(y), \qquad \mathcal{F}: \RR^d \to \RR^d.
\end{displaymath}
A study of the order conditions of ARK (which are equivalent to GARK)
was carried out in \cite{araujo97smb}. The generalization to the GARK formalism was 
treated  in \cite{sandu15gark} mostly with focus on implicit-explicit methods.
A further treatment, especially with focus on DAEs and stiff systems, can be found in
\cite{tanner18gark}. All these methods use expansion in elementary
differentials and colored trees (N-trees). For completeness, we
summarize the order analysis in this section.

The order conditions for GARK methods 
can be derived in a similar
manner to those of standard RK methods. The trees that
define the order conditions are exactly those of RK methods, except
for the fact that one has to consider all the possible combinations of
colors $1, \ldots, N$ associated to each of the vector field. The set
NT of $N$-trees consists of all Butcher trees with colored
vertices. The order of a N-tree $u\in \NT$, denoted as $\rho(u)$, is
the number vertices in $u$. The empty tree is denoted as $\emptyset$
and to emphasize that a N-tree $u$ has root of color $\nu$, we will
write $u^{[\nu]}$. Note that $u^{[\nu]}=[u_1,\ldots, u_m]^{[\nu]}$, where
$u_1, \ldots, u_m$ are the non-empty N-subtrees obtained removing the
root of $u$. The set of N-trees with root $\nu$ is denoted by
$\NT_\nu$. 
As in the setting of Butcher trees, we will denote by $\sigma(u) $ the
number of symmetries of $u \in \NT$ and by $\gamma(u)$ its density,
which is defined in a recursive manner as 
\begin{eqnarray*}
  \gamma(\emptyset) &=& 1 \\
  \gamma(\tau^{[\nu]}) &=& 1, \qquad \nu = 1, \ldots, N, \\
  \gamma(u) &=& \rho(u) \gamma(u_1) \cdots \gamma(u_m), \qquad u = u^{[\nu]}=[u_1,\ldots, u_m]^{[\nu]},
\end{eqnarray*}
where $\tau^{[\nu]}$ denotes the single vertex of color $\nu$.

For each tree $u \in \NT$ there is an elementary differential
$F(u):\RR^d\to \RR^d$ associated to it.
Elementary differentials are multilinear maps and are recursively
defined for each component $i=1, \ldots, d$ of the vector field
$\mathcal{F}$ as
\begin{eqnarray*}
  F^i(\emptyset) (y)&=& y^i \\
  F^i(\tau^{[\nu]}) (y) &=& \mathcal{F}^{\{\nu\}, i}(y), \qquad \nu = 1,
                       \ldots, N \\
  F^i(u)(y) &=& \sum_{i_1, \ldots, i_m =1}^d\frac{\de^m
                \mathcal{F}^{\{\nu\}, i}}{\de y^{i_1} \cdots \de
                y^{i_m}}(y) F^{i_1}(u_1)(y) \cdots F^{i_m}(u_m)(y),
                \qquad u = u^{[\nu]}=[u_1,\ldots, u_m]^{[\nu]}.
\end{eqnarray*}
Thus, defining $\mathbf{c}:\NT \to \RR$, a mapping assigning to each
N-tree a real number, the exact solution $y(t+h)$ can be written as a formal power
expansion,
\begin{displaymath}
  y(t+h) = \mathrm{NB}(\mathbf{c}, y(t))= \sum_{u\in \NT}
  \frac{h^{\rho(u)}}{\sigma(u)} \mathbf{c}(u) F(u) (y(t)),
\end{displaymath}
with 
\begin{displaymath}
  \mathbf{c}(u) = \frac{1}{\gamma(u)}.
\end{displaymath}

A similar expansion holds for the numerical method,
\begin{displaymath}
  y_{n+1} = \mathrm{NB}(\mathbf{d}, y_n).
\end{displaymath}
Thus: 
\begin{theorem}[Order of GARK methods, \cite{sandu15gark}]
  \label{th:order_cond}
  A GARK method is of order $r$ iff
  \begin{displaymath}
    \mathbf{d}(u) = \frac1{\gamma(u)} 
  \end{displaymath}
  for all $u \in \NT$ with $1\leq \rho(u) \leq r$.
\end{theorem}
We sketch the main moments of the proof. The mapping $\mathbf{d}: \NT
\to \RR$ depends on the GARK method and 
is also defined in a recursive manner using the internal stages of the
method.
One has
\begin{eqnarray*}
  Y_i^{\{l\}}& =& \mathrm{NB} (\mathbf{d}_i^{\{l\}}, y_n) \\
 h \mathcal{F}^{\{m\} } (Y^{\{m\}}_i) &=&
                                          \mathrm{NB}(\mathbf{g}_i^{\{m\}}, y_n),
\end{eqnarray*}
hence
\begin{displaymath}
  \mathbf{d}(u) = \sum_{l=1}^N \sum_{i=1}^{\myss{l}} \bb{l}_i
  \mathbf{g}_i^{\{l\}} (u), \qquad u \in \NT-\{\emptyset\},
\end{displaymath}
and, similarly,
\begin{displaymath}
  \mathbf{d}_i^{\{l\}}(u) = \sum_{m=1}^N \sum_{j=1}^{\myss{m}}
  \aa{l,m}_{i,j} \mathbf{g}_j^{\{m\}} (u).
\end{displaymath}
Using these in a recursive manner, one finds that, for $u = [u_1,
\ldots, u_l]^{[\nu]}$, 
\begin{displaymath}
  \mathbf{g}_i^{\{m\}} (u) = \delta_{\nu,m} \sum_{n_1, \ldots, n_l}
  \sum_{j_1, \ldots, j_l} \aa{m,n_1}_{i,j_1} \cdots \aa{m,n_l}_{i,
    j_l} \mathbf{g}_{j_1}^{\{ n_1\}}(u_1) \cdots \mathbf{g}_{j_l}^{\{n_l\}}(u_l),
\end{displaymath}
where $\delta_{\nu, m} = 1$ for $\nu=m$ and zero otherwise, implying
that $\mathbf{g}_i^{\{m\}} (u) =0$ whenever $u
=\emptyset$ is the empty tree or $u$ has root $\nu\not=m$. When $u =
\tau^{[m]}$, then $\mathbf{g}_i^{\{m\}} (u) =1$.

In this paper, we proceed in a manner similar to
\cite{bornemann01rkm}. We introduce the 
following notation:
\begin{eqnarray}
  \mathcal{A}^{[u]} &=& \ones^{\{l\}} \qquad \hbox{for } u =
                        \tau^{[l]}, \label{eq:Ataul}\\
  \mathcal{A}^{[u]} &=& (\AA{l, \nu_1} \mathcal{A}^{[u_1]})\odot
                          \cdots \odot (\AA{l, \nu_m}
                          \mathcal{A}^{[u_m]})\qquad \hbox{for }
                        u=[u_1, \ldots, u_m]^{[l]}
                        \label{eq:Au}
\end{eqnarray}
where  $\nu_1, \ldots, \nu_m$ are the root
colors of the subtrees $u_1, \ldots, u_m$ respectively and $\odot$ is
the componentwise vector multiplication.
Thus, if $ u=[u_1, \ldots, u_m]^{[l]}$, then
\begin{displaymath}
  \mathbf{d}(u) = \sum_{i=1}^{\myss{l}} \bb{l}_i \mathcal{A}^{[u]}_i
  = {\bb{l}}^T \mathcal{A}^{[u]}.
\end{displaymath}
Denoting by $\NT_l$ the set of N-trees with root $l$ and matching for
all the roots $l\in \{1, \ldots, N\}$, the numerical method can be written as
\begin{displaymath}
  y_{n+1} = y_n + \sum_{k=1}^r h^k \sum_{l=1}^N \sum_{\substack{
      u \in \NT_l\\\rho(u)=k}}\frac{1}{\sigma(u)}
     {\bb{l}}^T \mathcal{A}^{[u]} F(u)  + O(h^{r+1}). 
   \end{displaymath}
Performing a similar ordering for the trees in the exact solution, the
order conditions can be written elegantly as
     \begin{equation}
       \label{eq:bA}
       {\bb{l}}^T \mathcal{A}^{[u]} =\frac{1}{\gamma(u)}, \qquad  u\in
     \NT_{l}, \quad l = 1, \ldots, N, 
     \end{equation}
   for trees of order $\rho(u)=1, \ldots, r$.

   \begin{table}[ht]
     \centering
     \begin{tabular}{|c|c|c|c|l|}
       \hline 
       RK-tree $u$ & Order $\rho(u)$ &$\gamma(u)$& GARK order condition &
                                                            $\in\{1,\ldots, N\}$ \\ \hline
        $\ab$   & 1      & 1& $ {\bb{l}}^T\ones^{\{l\}} =1$ & $\forall
                                                          l$\\[5pt] \hline
       $\aabb$ & 2     & 2 & ${\bb{l}}^T \AA{l, m} \ones^{\{m\}}=\frac12$ & $\forall
                                                          l,m$\\[5pt] \hline
       $\aababb$& 3  & 3 & ${\bb{l}}^T ((\AA{l,m}\ones^{\{m\}})
                           \odot
                           (\AA{l,n}\ones^{\{n\}}))=\frac13$
                & $\forall l,m,n$\\[5pt] \hline
       $\aaabbb$ & 3 & 6 & ${\bb{l}}^T \AA{l,m} \AA{m,n}
                           \ones^{\{n\}} = \frac16$ & $\forall
                                                           l,m,n$\\[5pt]
       \hline
       $\aabababb$ & 4 & 4& ${\bb{l}}^T ((\AA{l,m}\ones^{\{m\}})
                           \odot
                           (\AA{l,n}\ones^{\{n\}})\odot  (\AA{l,u}\ones^{\{u\}}))=\frac14$
                & $\forall l,m,n,u$\\[5pt] \hline
       $\aaabbabb$ & 4& 8 & ${\bb{l}}^T((\AA{l,m}
                            \AA{m,n}\ones^{\{n\}})\odot
                            (\AA{l,u}\ones^{\{u\}}))=\frac18$ & $\forall l,m,n,u$\\[5pt] \hline
       $\aaababbb $ & 4 & 12 & ${\bb{l}}^T
                               \AA{l,m}((\AA{m,n}\ones^{\{n\}})\odot (\AA{m,u}\ones^{\{u\}}))=\frac1{12}$&$\forall l,m,n,u$\\[5pt] \hline
       $\aaaabbbb$ & 4 & 24 & ${\bb{l}}^T
                              \AA{l,m}\AA{m,n}\AA{n,u}
                \ones^{\{u\}}=\frac1{24}$& $\forall l,m,n,u$\\[5pt] \hline
       
     \end{tabular}
     \caption{Order conditions for GARK methods up to order 4. The
       symbol $\odot$ denotes componentwise vector multiplication.}
   \end{table}

We note that if the diagonal methods $(\AA{l,l},
\bb{l},\cc{l})$ have order $r^{\{l\}}$, then the corresponding order
conditions for $l=m=n=u=\ldots$ (all the indices equal) are
satisfied up to order $r^{\{l\}}$, as they are the same as the
underlying order conditions for the RK method.
Moreover, for consistency, it is also reasonable to require that
\begin{equation}
  \label{eq:Aone}
  \AA{l,m} \ones^{\{m\}}= \cc{l},\qquad \forall l,m.
\end{equation}
This condition is automatically satisfied for $m=l$, as long as the
underlying RK method $(\AA{l,l},\bb{l},\cc{l})$ is consistent. Using
the consistency condition \R{eq:Aone}, one recovers exactly the order
condition listed in \cite{sandu15gark}.

\begin{definition}[Simplifying conditions for GARK methods,
  \cite{tanner18gark}]
  \label{def:symplifying GARK}
  Simplifying conditions for GARK methods: for $l,m=1, \ldots N$,
  where $N$ is the number of methods,
  \begin{align}
    \label{eq:B}
    B^{\{l\}}(\rl{l}) :\qquad & \sum_{i=1}^{\myss{l}} \bb{l}_i( \cc{l}_i)^{k-1} = \frac 1k,
                        \qquad k=1, \ldots, r^{\{l\}}
    \\
    \label{eq:C}
    C^{\{l,m\}} (\etalm{l,m}): \qquad &\sum_{j=1}^{\myss{m}} \aa{l,m}_{i,j} (\cc{m}_j)^{k-1} =
    \frac{(\cc{l}_i)^k}k, \qquad i=1, \ldots, \myss{l}, \quad k = 1, \ldots,\etalm{l,m}.
    \\
    \label{eq:D}
    D^{\{l,m\}}(\zetalm{l,m}): \qquad & \sum_{i=1}^{\myss{l}} \bb{l}_i (\cc{l}_i)^{k-1} \aa{l,m}_{i,j} =
    \frac{\bb{m}_j}k (1-(\cc{m}_j)^k), \qquad j = 1, \ldots, \myss{m}, \quad k =1,
    \ldots, \zetalm{l,m}.
  \end{align}
\end{definition}
Note that $ C^{\{l,m\}} (1)$ is the same as \R{eq:Aone}, required for consistency.
The above conditions are generalization of the corresponding $B,C,D$
conditions for RK methods, which are recovered when $N=1$.
The $B$ condition \R{eq:B} implies that that the quadrature formula
with weights $b_i^{\{l\}}$ and nodes $c_i^{\{l\}}$ has order $p^{\{
  l\}}$ and, provided that the 
diagonal methods are consistent so that \R{eq:Aone} holds, it is
equivalent to state that the order conditions for the bushy trees like $u=\aabababb$ with $k$ vertices ($\gamma(u)=k$) are
automatically satisfied up to $p^{\{ l\}}$. The $C$ condition \R{eq:C} is related to the
notion of the \emph{stage order} of the method, that is the order of
approximation at the internal stages. The $D$ condition \R{eq:D} is a
simplifying condition, that guarantees the order conditions for
trees of type $u=\aaababababbb$ are also satisfied.

The theorem below generalizes an important theorem due to Butcher, who
used the RK simplifying assumptions $B,C,D$ to 
obtain an estimate of the order of the underlying method.

\begin{theorem}[Simplifying Assumption Theorem, \cite{tanner18gark}] 
  \label{th:orderGARK}
  If $\Bl{l}(\rl{l})$, $\Clm{l,m}(\etalm{l,m})$,
  $\Dlm{l,m}(\zetalm{l,m})$ are satisfied for $l,m=1,\ldots,N$, then the
  order of the GARK method is at least
  \begin{displaymath}
    \min\{r, 2\eta+2, \xi+1, \eta+\zeta+1\},
  \end{displaymath}
  where $r = \min_{l} \{\rl{l}\}$,  $\eta = \min_{l,m}
  \{\etalm{l,m}\}$, $\zeta=\min_{l,m} \{\zetalm{l,m}\}$ and $\xi =
  \min_{l,m} \{ \etalm{l,m}+\zetalm{l,m}\}$.
\end{theorem}

Definition~\R{def:symplifying GARK} and Theorem~\R{th:orderGARK} are
originally stated for $N=2$ but extention to a general $N$-terms case
is immediate.

\section{Some techniques to construct the transfer matrices
  $\AA{l,m}$, for $l\not=m$}
\label{sec:transfer}
As the variational methods discussed in this paper are a more general
case of the methods proposed in \cite{zanna19sps},  we generalize the
approach in \cite{zanna19sps} for the construction of the coefficients and
use rather the order conditions to establish  the order of the resulting GARK method.

Assume the primary (diagonal) methods, $(\AA{m,m},\bb{m},\cc{m})$,
$m=1,\ldots, N$, are given. Denote by
$\LL_k^{\{m\}}(t)=\prod_{j=1, j\not=k}^{\myss{m}}\frac{t-\cc{m}_j}{\cc{m}_k-\cc{m}_j}$
the $k$th Lagrange interpolating polynomial based on 
the $\cc{m}$ nodes that we assume to be distinct. We wish to construct
$\AA{l,m}$, corresponding to the $\cc{l}$ nodes (also distinct).

\subsection{Collocation}
\label{sec:coll}
An obvious choice of the coefficients is by
a technique similar to \emph{collocation}, i.e.\ by taking
\begin{equation}
  \label{eq:coll}
  \aa{l,m}_{i,j} = \int_{0}^{\cc{l}_i} \LL^{\{m\}}_j (\tau) \D
  \tau,\qquad l\not=m.
\end{equation}
If $\cc{l}=\cc{m}$ then $\AA{l,m}$ is the matrix of the collocation RK
method based on the quadrature $(\bb{l},\cc{l})$. However, in general, $\AA{l,m} \not=
\AA{m,m}$, unless $(\AA{m,m},\bb{m},\cc{m})$ is itself a collocation
RK method.

This collocation construction is equivalent to the construction for some SPARK
methods in \cite{jay98spark} and GARK methods in \cite{tanner18gark} developed
in the DAEs context.

\textbf{Example:} Consider $\AA{1,1}$ to be a Gauss-Legendre RK and
$\AA{2,2}$ to be a Lobatto IIIA method, both of order 4. The
corresponding GARK based on collocation is:
\begin{displaymath}
  \begin{array}{c|cccccc}
    \frac12-\frac16{\sqrt{3}}&\frac14&\frac14-\frac16\sqrt3&&\frac16-\frac1{108}\sqrt3&\frac13+\frac4{27}\sqrt3& \frac1{108}\sqrt3\\[2pt]
    \frac12+\frac16{\sqrt{3}}& \frac14+\frac16\sqrt3 & \frac14&& \frac16+\frac1{108}\sqrt3&\frac13-\frac4{27}\sqrt3& -\frac1{108}\sqrt3\\[8pt]
    0 & 0&0 &&0 & 0 & 0\\[2pt]
    \frac12&\frac14+\frac18\sqrt3&\frac14-\frac18\sqrt3&&\frac5{24} & \frac13 & -\frac1{24}\\[2pt]
    1 & \frac12&\frac12&&\frac16 & \frac23 & \frac16\\[2pt]
    \hline
    & \frac12 & \frac12&& \frac16 &\frac23 & \frac16
  \end{array}
\end{displaymath}
Its symplectic conjugate methods is 
\begin{displaymath}
\begin{array}{c|cccccc}
    \frac12-\frac16{\sqrt{3}}&\frac14&\frac14-\frac16\sqrt3&&\frac16&\frac13-\frac1{6}\sqrt3& 0\\[2pt]
    \frac12+\frac16{\sqrt{3}}& \frac14+\frac16\sqrt3 & \frac14&& \frac16&\frac13+\frac1{6}\sqrt3& 0\\[8pt]
    0 & \frac1{36}\sqrt3&-\frac1{36}\sqrt3 &&\frac16 & -\frac16 & 0\\[2pt]
    \frac12&\frac14+\frac19\sqrt3&\frac14-\frac19\sqrt3&&\frac16 & \frac13 & 0\\[2pt]
    1 & \frac12&\frac12&&\frac16 & \frac56 & 0\\[2pt]
    \hline
    & \frac12 & \frac12&& \frac16 &\frac23 & \frac16
\end{array}
\end{displaymath}

\subsection{Interpolation}
\label{sec:interp}
Another choice is \emph{interpolation} for the $l$-variables using an
interpolating polynomial based on the method $\AA{m,m}$ for the ${m}$
variables. Interpolating the results given by the $m$ method on the
$\cc{l}$ nodes gives
\begin{equation}
  \label{eq:interp}
  \AA{l,m} = \LL^{\{m\}}(\cc{l}) \AA{m,m}, 
\end{equation}
where $\LL^{\{m\}}(\cc{l})$ is the $\myss{l} \times \myss{m}$ matrix with element $(i,j)$ given
as $\LL_j^{\{m\}}(\cc{l}_i)$. If $\cc{l}=\cc{m}$,
$\LL^{\{m\}}(\cc{l})=I$, and $\AA{l,m} = \AA{m,m}$.

\textbf{Example:} Let $\AA{1,1}, \AA{2,2}$ as above. The
corresponding GARK based on interpolation is:
\begin{displaymath}
  \begin{array}{c|cccccc}
    \frac12-\frac16{\sqrt{3}}&\frac14&\frac14-\frac16\sqrt3&&\frac16-\frac1{36}\sqrt3&\frac13-\frac1{9}\sqrt3& -\frac1{36}\sqrt3\\[2pt]
    \frac12+\frac16{\sqrt{3}}& \frac14+\frac16\sqrt3 & \frac14&& \frac16+\frac1{36}\sqrt3&\frac13+\frac1{9}\sqrt3& \frac1{36}\sqrt3\\[8pt]
    0 & \frac1{12}\sqrt3&-\frac1{12}\sqrt3 &&0 & 0 & 0\\[2pt]
    \frac12&\frac14+\frac1{12}\sqrt3&\frac14-\frac1{12}\sqrt3&&\frac5{24} & \frac13 & -\frac1{24}\\[2pt]
    1 & \frac12+\frac1{12}\sqrt3&\frac12-\frac1{12}\sqrt{3}&&\frac16 & \frac23 & \frac16\\[2pt]
    \hline
    & \frac12 & \frac12&& \frac16 &\frac23 & \frac16
  \end{array}
\end{displaymath}
Similarly, its symplectic conjugate is
\begin{displaymath}
  \begin{array}{c|cccccc}
    \frac12-\frac16{\sqrt{3}}&\frac14&\frac14-\frac16\sqrt3&&\frac16-\frac1{36}\sqrt3&\frac13-\frac1{9}\sqrt3& -\frac1{36}\sqrt3\\[2pt]
    \frac12+\frac16{\sqrt{3}}& \frac14+\frac16\sqrt3 & \frac14&& \frac16+\frac1{36}\sqrt3&\frac13+\frac1{9}\sqrt3& \frac1{36}\sqrt3\\[8pt]
    0 & \frac1{12}\sqrt3&-\frac1{12}\sqrt3 &&\frac16 & -\frac16 & 0\\[2pt]
    \frac12&\frac14+\frac1{12}\sqrt3&\frac14-\frac1{12}\sqrt3&&\frac16 & \frac13 & 0\\[2pt]
    1 & \frac12+\frac1{12}\sqrt3&\frac12-\frac1{12}\sqrt{3}&&\frac16 & \frac56 & 0\\[2pt]
    \hline
    & \frac12 & \frac12&& \frac16 &\frac23 & \frac16
  \end{array}
\end{displaymath}
The methods in \S~\ref{sec:coll}-\ref{sec:interp} all have order 4,
as it can be checked directly from the order conditions in Section~\ref{sec:order}.

\subsection{Consistency and simplifying conditions for transfer
  matrices by collocation and interpolation}
In this section we assume the diagonal methods $(\AA{l,l}, \bb{l},
\cc{l})$, $l=1,\ldots, N$, given, and study how
the construction of the transfer matrices, by collocation or
interpolation, contributes to the simplifying assumptions.

\begin{prop}
  Assume that the diagonal methods are consistent. 
 Then the matrices $\AA{l,m}$ in \R{eq:coll} and \R{eq:interp} satisfy
 $\AA{l,m}\oneslm{m} = \cc{l}$, as 
 \begin{equation}
   \label{eq:lincomb}
  \AA{l,m}=\tau \AA{l,m}_{\mathrm{interp}} + (1-\tau)
  \AA{l,m}_{\mathrm{colloc}}.
\end{equation}
\end{prop}
\begin{proof}
For the collocation-coefficients: taking the $i$th component, $\sum_{j} \aa{l,m}_{i,j} = \sum_j
\int_0^{\cc{l}_i}\LL_j^{\{m\}} (\tau) d \tau=\cc{l}_i$, by exchanging the order
of sum and integration and using the fact that $\sum_{j}\LL^{\{m\}}_j(t) = 1$. For the
interpolation-type coefficients, $\AA{l,m} \oneslm{m} =
\LL^{\{m\}}(\cc{l}) \AA{m,m} \oneslm{m} = \LL^{\{m\}}(\cc{l}) \cc{m}$,
because of the assumption $\AA{m,m}\ones^{\{m\}}
= \cc{m}$. The last passage, $\LL^{\{m\}}(\cc{l}) \cc{m}=\cc{l}$
follows from the the fact that $\sum_{i=1}^m\LL^{\{m\}}_i(t) \cc{m}_i =t$ (identity
function), hence evaluations in the $\cc{l}$ does the trick.
\end{proof}
The consequence of the above result is that any choice
\R{eq:lincomb}, with transfer coefficients constructed as in this section,
will give a consistent GARK method.

In \cite{zanna19sps} we developed some results specific to the case
with $V^{\{1\}}(q,p) =p$, $V^{\{2\}}(q,p) =0$,
$F^{\{1\}}=F^{\{1\}}(q)$,  $F^{\{2\}}= F^{\{2\}}(q)$.
The following results generalize those findings with the notation of
this paper.

\begin{lemma}
\label{le:stageorder}
   Assume that $\AA{l,m}$ ($m\not=l$) is constructed either by
   interpolation or collocation.  Then 
   \begin{equation}
     \label{eq:4}
     C^{\{l,m\}}(\etalm{l,m}): \qquad \sum_{j=1}^{\myss{m}} \aa{l,m}_{i,j} (\cc{m}_j)^{k-1} = \frac{(\cc{l}_i)^k}{k},
     \qquad  i=1, \ldots, \myss{l},
   \end{equation}
   for $k = 1, \ldots, \min\{\etalm{m,m},\myss{m}-1\}$ in the
   interpolation setting and $k = 1, \ldots, \myss{m}-1$ in the
   collocation setting.
 \end{lemma}
 \begin{proof}
   The proof is essentially the same as in \cite{zanna19sps}.  We simplify
   it and put it in the formalism of GARK methods.
   Consider the interpolation setting first.
   \begin{eqnarray*}
     \sum_{j=1}^{\myss{m}} \aa{l,m}_{i,j} (\cc{m}_j)^{k-1}  &=&
     \sum_{j=1}^{\myss{m}} \sum_{n=1}^{\myss{m}} \LL_n^{\{m\}}(\cc{l}_i)
     \aa{m,m}_{n,j} (\cc{m}_j)^{k-1} = \sum_{n=1}^{\myss{m}}
     \LL_n^{\{m\}}(\cc{l}_i)  \sum_{j=1}^{\myss{m}} \aa{m,m}_{n,j}
                                                              (\cc{m}_j)^{k-1}\\
     &=& \sum_{n=1}^{\myss{m}} \LL_n^{\{m\}}(\cc{l}_i)
         \frac{(\cc{m}_n)^k}{k} \qquad
         \hbox{from $C^{\{m,m\}}(\etalm{m,m})$,} \quad k=1,\ldots,
         \etalm{m,m}\\
     &=& \frac{(\cc{l}_i)^{k}}{k} \qquad k= \min\{\etalm{m,m}, \myss{m}-1\}
   \end{eqnarray*}
   where last passage follows from the fact that $\sum_{n=1}^{\myss{m}}
   \LL^{\{m\}}_n(t) f_n$ is the interpolating polynomial based on
   $\myss{m}$ nodes and is exact for polynomials up do degree $\myss{m}-1$.

   Next, for the collocation setting,
   \begin{eqnarray*}
     \sum_{j=1}^{\myss{m}} \aa{l,m}_{i,j} (\cc{m}_j)^{k-1} & =&
     \sum_{j=1}^{\myss{m}} \int_0^{\cc{l}_i} \LL^{\{m\}}_j(\tau)
     (\cc{m}_j)^{k-1} \D \tau = \int_0^{\cc{l}_i}   \sum_{j=1}^{\myss{m}}
     \LL^{\{m\}}_j(\tau)  (\cc{m}_j)^{k-1} \D \tau \\
      &=&  \int_0^{\cc{l}_i} \tau^{k-1} \D \tau
     = \frac{(\cc{l}_i)^k}k, \qquad k=1, \ldots, \myss{m}-1
   \end{eqnarray*}
    where the third passage follows from the fact that the unique
    interpolant to $(\cc{m}_i)^{k-1}$ with nodes $\cc{m}_i$ is the
    function $\tau^{k-1}$.
 \end{proof}
 \begin{coroll}
   Let $\AA{l,m}_{\mathrm{interp}}$ and $\AA{l,m}_{\mathrm{coll}}$ be
   the coefficients of the secondary method based on interpolation and
   collocation. Then the method
\begin{displaymath}
\AA{l,m}= \tau_1 \AA{l,m}_{\mathrm{interp}}+\tau_2
\AA{l,m}_{\mathrm{coll}}, \qquad \tau_1 + \tau_2=1,
\end{displaymath}
 satisfies $C^{\{l,m\}}(\etalm{l,m})$, where $\etalm{l,m}=
 \min\{\etalm{m,m}, \myss{m}-1\}$, for any $\tau \in \RR$.
\end{coroll}
There is no immediate way to extend the above results to the
symplectic conjugate method. However, if the conjugate method also satisfies $C^{\{l,m\}}
(\etalm{l,m})$, then so does $\AA{l,m}= \tau_1 \AA{l,m}_{\mathrm{interp}}+\tau_2
\AA{l,m}_{\mathrm{coll}}+\tau_3 \wAA{l,m}_{\mathrm{interp}}+\tau_4
\wAA{l,m}_{\mathrm{coll}}$ when $\tau_1 + \cdots+ \tau_4=1$.

The $D^{\{l,m\}}$ properties are more difficult to generalize
\cite{tanner18gark}. However, it is reasonable to expect that, with
both transfer matrix as in \S~\ref{sec:coll}-\ref{sec:interp}, the
resulting method will have order equal to the
minimum of the order of the diagonal methods. This is not obvious for
superconvergent diagonal methods (like Gauss-Legendre, Lobatto, etc.)
but the constructions in \S~\ref{sec:coll}-\ref{sec:interp} seem to
indicate that it is generally true. The same was observed for some
methods with transfer matrices constructed by collocation in \cite{tanner18gark}.

 \section{Special cases for $N=2$}
\label{sec:special-case}
 We consider two main special cases, the first one when
 \begin{displaymath}
   L(q,\dot q) = L^{\{1\}} (q, \dot q) +  L^{\{2\}} (q), \qquad H(q,p)
   =  H^{\{1\}} (q, p) +  H^{\{2\}} (q)
 \end{displaymath}
 that we call a \emph{semi-separable} case, and the fully
 \emph{separable} case
 \begin{displaymath}
    L(q,\dot q) = \frac12 \dot q^T \dot q -U(q), \qquad H(q, p) =
    \frac12 p^T p + U(q),
  \end{displaymath}
  for a potential function $U$. Some examples of the
  semi-separable case are mechanical systems in which the mass matrix
 depends on the generalized variable $q$, while $L^{\{2\}} (q)=-U(q)$, with $U$ a
 potential depending on the generalized variable $q$ only. Other
 example are the case of systems with \emph{holonomic} constraints or
 stiff systems \cite{jay98spark, tanner18gark, zanna19sps}. In the 
 separable case, we assume that the mass matrix is constant and
 invertible, and, without loss of generality, we set it equal to $I$.

 \subsection{The semi-separable case}
The variational derivation is the same as in \cite{zanna19sps}, see also
Section~\ref{sec:vari-deriv}. Let
\begin{displaymath}
   L(q,\dot q) = L^{\{1\}} (q, \dot q) +  L^{\{2\}} (q).
 \end{displaymath}
One has
\begin{eqnarray*}
   \PP{1}_j = \frac{\de L^{\{1\}}}{\de \dot q} (\QQ1_j,
     \dot{Q}^{\{1\}}_j) & &\dot{P}^{\{1\}}_j=  \frac{\de L^{\{1\}}}{\de q} (\QQ1_j,
     \dot{Q}^{\{1\}}_j)\nonumber \\
 \PP{2}_j = \frac{\de L^{\{2\}}}{\de \dot q} (\QQ2_j)=0 && \dot{P}^{\{2\}}_j=  \frac{\de L^{\{2\}}}{\de q} (\QQ2_j)\nonumber 
\end{eqnarray*}
The constraint condition reads
\begin{displaymath}
  \bb{1}_j \PP1_j = \bb1_j \lambda - h\sum_{i=1}^{\myss1} \bb1_i
  \aa{1,1}_{i,j} \dot{P}^{\{1\}}_j - h \sum_{i=1}^{\myss2} \bb2_i
  \aa{1,2}_{i,j} \dot{P}^{\{2\}}_j .
\end{displaymath}
As there is no $\PP2$, the above relation \emph{requires only} $\bb1_j$
to be nonzero, $j =1\ldots, \myss1$.
 The symplecticity conditions are
 \begin{eqnarray*}
   \wAA{1,i} = (\oneslm{1,i} - (B^{\{1\}})^{-1} (\AA{i,1})^T B^{\{i\}}
   , \qquad i=1,2.
 \end{eqnarray*}

Since $\PP2$ is not defined, the
$\dot q$ variables are recovered simply by inverting the Legendre
transform $\PP{1}_j = \frac{\de L^{\{1\}}}{\de \dot q}
(\QQ1_j,\dot{Q}^{\{1\}}_j) $, implying that the $q$ variable is only direct
function of $\PP1$, and therefore there is no need for $\AA{l,2}$, for
$l=1,2$. Eventually, the tables can be completed by choosing an
appropriate $\AA{2,2}$ compatible with the nodes $\cc2$ and the
weights $\bb2$. Such matrix $\AA{2,2}$ can also be useful in
constructing the $\AA{2,1}$ coefficients, if, for instance, one uses
the methods proposed in Section~\ref{sec:transfer}.
The method, in a partitioned GARK formalism reads
\begin{displaymath}
  Q: \qquad
  \begin{array}{c|c}
    \cc1& \AA{1,1} \\
    \cc2& \AA{2,1}\\ \hline & \bb1
  \end{array},
  \qquad  \qquad
  P : \begin{array}{c|cc}
        \cc1 & \wAA{1,1} & \wAA{1,2} \\ \hline
        &\bb1 & \bb2,
   \end{array}     
 \end{displaymath}
 where, in the Hamiltonian formalism introduced earlier,
 \begin{eqnarray*}
   V^{\{1\}}(\QQ1_j,\PP1_j) = \dot {Q}^{\{1\}}_j, &&  V^{\{2\}}=0 ,\\
   F^{\{1\}}(\QQ1_j,\PP1_j) = \dot {P}^{\{1\}}_j     &&  F^{\{2\}}(\QQ2_j) = \dot {P}^{\{2\}}_j ,
 \end{eqnarray*}
 the first line being obtained by inverting the Legendre transform for
 $\dot {Q}^{\{1\}}_j$.
 In the Hamiltonian formalism, the methods read
 \begin{eqnarray*}
   \QQ{1}_i &=& q_0 + h \sum_{j=1}^{\myss1} \aa{1,1}_{i,j}V^{\{1\}}(\QQ1_j, \PP1_j)\\
   \QQ{2}_i &=& q_0 + h \sum_{j=1}^{\myss1}
                \aa{2,1}_{i,j}V^{\{1\}}(\QQ1_j, \PP1_j)\\
   q_1 &=& q_0 + h \sum_{i=1}^{\myss1} \bb1_i V^{\{1\}}(\QQ1_i, \PP1_i) \\
   \PP1_i &=& p_0 + h \sum_{j=1}^{\myss1} \waa{1,1}_{i,j}
              F^{\{1\}}(\QQ1_j,\PP1_j) + h \sum_{j=1}^{\myss2} \waa{1,2}_{i,j}
              F^{\{2\}}(\QQ2_j) \\
   p_1 &=& p_0 + h \sum_{i=1}^{\myss1} \bb1_i  F^{\{1\}}(\QQ1_i,\PP1_i) + h \sum_{i=1}^{\myss2} \bb2_i  F^{\{2\}}(\QQ2_i)
 \end{eqnarray*}

\subsection{The fully separable case}
\label{sec:separable-case}
Proceding as above, 
\begin{eqnarray*}
   \PP{1}_j = \frac{\de L^{\{1\}}}{\de \dot q} (\dot{Q}^{\{1\}}_j) &
  &\dot{P}^{\{1\}}_j=  0 \nonumber \\
 \PP{2}_j = 0 && \dot{P}^{\{2\}}_j=  \frac{\de L^{\{2\}}}{\de q} (\QQ2_j).\nonumber 
\end{eqnarray*}
The constraint condition reads
\begin{displaymath}
  \bb{1}_j \PP1_j = \bb1_j \lambda - h \sum_{i=1}^{\myss2} \bb2_i
  \aa{1,2}_{i,j} \dot{P}^{\{2\}}_j ,
\end{displaymath}
where, as before, we require the $\bb1_j$ to be nonzero.
In the Hamiltonian setting, taking
\begin{displaymath}
  \dot Q^{\{1\}}_i= V^{\{1\}}(\PP1_i), \qquad \dot P^{\{2\}}_i= F^{\{2\}}(\QQ2_i),
\end{displaymath}
the system reads
\begin{eqnarray*}
  \QQ2_i &=& q_0 + h \sum_{j=1}^{\myss1}\aa{2,1}_{i,j}
             V^{\{1\}}(\PP1_j),\\
  q_1 &=& q_0 + h \sum_{i=1}^{\myss1} \bb1_i V^{\{1\}}(\PP1_i)\\
  \PP1_i &=& p_0 + h \sum_{j=1}^{\myss2} \waa{1,2}_{i,j}F^{\{2\}}(\QQ2_j)\\
  p_1 &=& p_0 + h \sum_{i=1}^{\myss2}  \bb2_i F^{\{2\}}(\QQ2_i).
\end{eqnarray*}
The scheme has tableau
\begin{displaymath}
  Q: \qquad
  \begin{array}{c|c}
    \cc2& \AA{2,1}\\ \hline & \bb1
  \end{array},
  \qquad  \qquad
  P : \begin{array}{c|c}
        \cc1 & \wAA{1,2} \\ \hline
        & \bb2
   \end{array}     
 \end{displaymath}
and the matrices need not be square, as in the usual
setting. For instance one could choose the symplectic method
\begin{equation}
  \label{eq:rectangular}
  Q: \quad
  \begin{array}{c|cc}
    0&0&0\\
    \frac12 & \frac14+\frac18 \sqrt3 & \frac14-\frac18 \sqrt3\\[2pt]
    1 & \frac12 & \frac12\\[2pt] \hline
     & \frac12 & \frac12
  \end{array} \qquad \qquad
  P:\quad
  \begin{array}{c|ccc}
    \frac12-\frac16\sqrt3&\frac16&\frac13-\frac16\sqrt3&0\\[2pt]
    \frac12+\frac16\sqrt3&\frac16&\frac13+\frac16\sqrt3&0\\[2pt]\hline
    & \frac16 & \frac23 & \frac16
  \end{array}
\end{equation}
see Section~\ref{sec:transfer} (collocation example).
We observe that it is allowed to use different weights $\bb1$ and
$\bb2$ for the $Q$ and the $P$ part respectively. 
The symplecticity condition is
\begin{eqnarray*}
   \wAA{1,2} = (\oneslm{1,2} - (B^{\{1\}})^{-1} (\AA{2,1})^T )B^{\{2\}},
\end{eqnarray*}
which coincides with the well known sufficient symplecticity condition for PRK
for the separable case,
\begin{displaymath}
  \bb1_i \bb2_j = \bb1_i \waa{1,2}_{i,j} + \bb2_j \aa{2,1}_{j,i},
\end{displaymath}
\cite{hairer06gni, sanz-serna94nhp, araujo97smb}.
Note, however, that there must be some form of compatibility as the weights of
the $q$-part must be compatible with the nodes of the $p$-part and
viceversa.
The rectangular scheme \R{eq:rectangular} is essentially implicit only
in the $Q_2$ stage only, it has order four, as it follows from our
error analysis, shares many of the benefits of Gauss-Legendre and
Gauss-Lobatto methods, in addition to being slightly computationally less
expensive than both the underlying methods methods.

\begin{figure}[th]
  \centering
  \includegraphics*[width=.65\textwidth]{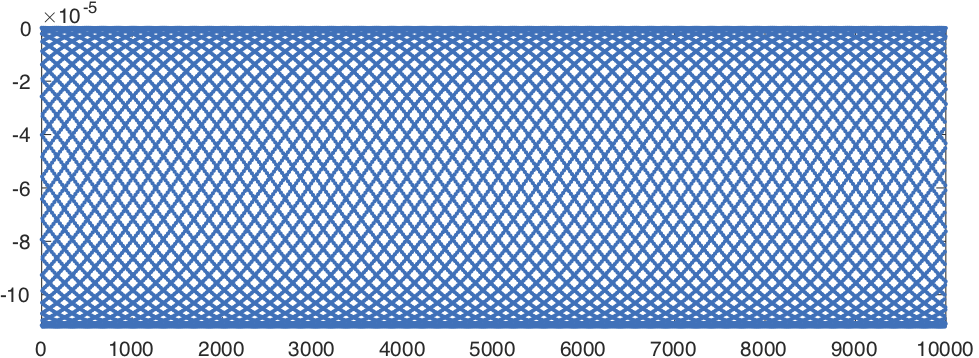} \quad
   \includegraphics*[width=.30\textwidth]{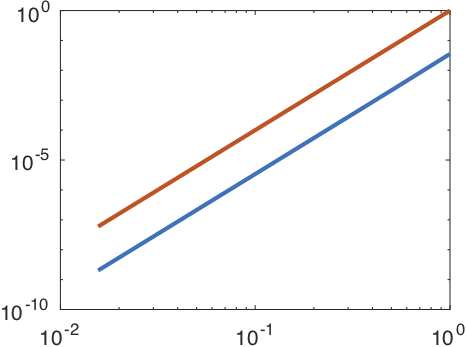}
  \caption{Error in the total energy (left) in $[0, 10^4]$ with
    stepsize $h=\frac12$ and  log-log order plot (right) for the
    unconventional method \R{eq:rectangular} applied to the harmonic oscillator. The red line is
    the reference line for order four. }
\end{figure}

In the general setting (non-separable) the weights used in the GARK
method for the $Q$ and $P$ part must be the same set of weights, as it
is clear from the variational derivation in Section~\ref{sec:vari-deriv}.

\section{Conclusions and further remarks}
\label{sec:concl}
We have presented a variational framework that allows for the
treatment of different Lagrangian terms by different RK methods. The
framework uses GARK methods and leads to symplectic partitioned
numerical methods that can be put in a GARK formalism. Further, we
have proposed two ways of constructing transfer coefficients
$\AA{l,m}$, which, in turn, allow for an infinite family of
coefficients satisfying the same order conditions.
The main idea is to use as diagonal methods one's favourite methods,
with good properties, and to derive transfer conditions that share
these properties and possibly other desirable properties, by tuning
the free parameters, see for instance \cite{zanna19sps}.

We have also discussed in more detail the special case $N=2$, for
semiseparable and separable problems. For the separable case, we have given an
example of an unconventional scheme, constructed from Gauss--Legendre
and Gauss--Lobatto,  with rectangular matrices. The method is
symplectic, fourth order, and essentially only implicit in one variable.

The following topics were not addressed in this paper, as they are
more natural in a Hamiltonian setting rather than Lagrangian setting. 
\begin{itemize}
\item The case of zero weights, as we have assumed $\bb{m}_i\not=0$,
  $i=1, \ldots, \myss{m}$, for all $m = 1, \ldots,N$. 
\item The case of redundant stages.  In several proofs we use the
  uniqueness and the order of Lagrange interpolation, and for this purpose we have assumed $\cc{m}_i\not =
  \cc{m}_j$, $i\not= j$, $i, j=1, \ldots, \myss{m}$, for all $m = 1,
  \ldots,N$. When some of the nodes coincide, extra conditions on
  derivatives might be required. 
\item The number of order conditions for an arbitrary symplectic partitioned GARK
  method for Hamiltonian vector fields.
  If it is desirable to derive \emph{all} the coefficients of a
  symplectic partitioned GARK method  for both the $q$ and $p$ variables imposing symplecticity of the vector
  field, there is a reduction in number of total order conditions due
  to the fact that the underlying vector field is not arbitrary, hence
  some of the order conditions fall out and need not be satisfied, see
  also  \cite{araujo97smb} .
\end{itemize}
These  three topics are more relevant in the context of the derivation
of new coefficients (also for the diagonal methods) directly from the order
conditions, rather than using one's favourite diagonal
methods, that already has some desirable properties, as it has been
the main focus of this paper.

\section*{Acknowledgements}
The author would like to thank the Isaac Newton Institute for
Mathematical Sciences, Cambridge, for support and hospitality during
the programme ``Geometry, Compatibility and Structure Preservation in
Computational Differential Equations'' where work on this paper was
undertaken. This work was supported by EPSRC grant no.\ EP/R014604/1.

\bibliographystyle{alpha}
\bibliography{SPGARK.bib}

\end{document}